\documentclass[10pt]{amsart}
\usepackage{graphicx}
\usepackage[colorlinks, linkcolor=blue,  citecolor=blue, urlcolor=blue]{hyperref}
\usepackage{cleveref}
\usepackage{autonum}
    \usepackage{algorithm}
\usepackage{algorithmic}
\usepackage{adjustbox}
\usepackage{xcolor}
\usepackage{float}

\theoremstyle{Proposition}
\newtheorem{theorem}{Theorem}[section]
\newtheorem{lemma}[theorem]{Lemma}
\theoremstyle{definition}
\newtheorem{definition}[theorem]{Definition}
\theoremstyle{remark}
\newtheorem{remark}[theorem]{Remark}
\DeclareMathOperator{\argmin}{{argmin}}

\newcommand{\R}{\mathbb{R}}

\newcommand{\bO}{\mathcal{O}}

\newcommand{\E}[1]{\mathbb{E} \left [{#1}\right ]}

\graphicspath{{SGD_figures/}}

\begin{document}
\title{Stochastic Gradient Descent with Polyak's learning rate}
\author{Adam M. Oberman and Mariana Prazeres}
\date{\today}

\begin{abstract}
Stochastic gradient descent (SGD) for strongly convex functions converges at the rate $\bO(1/k)$.  However, achieving good results in practice requires tuning the parameters (for example the learning rate) of the algorithm.  In this paper we propose a generalization of the Polyak step size, used for subgradient methods,  to Stochastic gradient descent.  We prove a non-asymptotic convergence at the rate $\bO(1/k)$ with a rate constant which can be better than the corresponding rate constant for optimally scheduled SGD.  We demonstrate that the method is effective in practice, and on  convex optimization problems and on training deep neural networks, and compare to the theoretical rate. 
\end{abstract}
\thanks{research supported by:  
AFOSR FA9550-18-1-0167 (A.O.)}
\maketitle

\section{Introduction}
Stochastic Gradient Descent (SGD) \cite{robbins1951,kalman1960new} is a widely used optimization algorithm  due to its ubiquitous use in machine learning \cite{bottou2016optimization}.  Convergence rates are available in a wide setting \cite{lacoste2012simpler,bottou2016optimization,qian2019sgd}.  To achieve the optimal convergence rate requires using an algorithm with parameters, for example, a scheduled learning rate, which depends on knowledge of parameters of the function which are often not available.  In this article, we propose an adaptive time step method for SGD, based on a generalization of the Polyak time step for the subgradient method, for SGD.  We prove a that this method achivies non-asymptotic convergence rate which can have a better rate constant than the one for scheduled SGD.

 Consider a differentiable function $f:\mathcal{X} \to \R$ depending on parameters $x\in\mathcal{X}$. Then, the problem at hand,
 \begin{equation}
 \min_{x\in\mathcal{X}} f(x),
 \end{equation}
may be solved approximately using the SGD step
\begin{align}\label{ASGD}\tag{SGD}
	x_{k+1} 
	= x_k - h_k \nabla_{mb} f(x_k),
\end{align}
where $h_k > 0 $ is the learning rate and $\nabla_{mb} f$ is a stochastic gradient. 

SGD is the method of choice for large scale machine learning problems \cite{bottou1991stochastic}.  When SGD is combined with momentum  \cite{polyak1964some,nesterov2013introductory} empirical performance is improved, but there is still no theoretical justification for the improvement~\cite{kidambi2018insufficiency}.  Other popular stochastic optimization algorithms that use different forms of averaging or variance reduction are still being studied and developed \cite{bottou2016optimization}.  However, in this work, we focus on  SGD without momentum, \eqref{ASGD}. 

The optimal convergence rate for SGD  in the strongly convex case is $\bO(1/k)$ \cite{rakhlin2011making, agarwal2009information, raginsky2011information}.  A learning rate schedule that achieves such rate is of the form $h_k = \bO(1/k)$, see \eqref{LR} in Section \ref{sec: asgd}. However, achieving the optimal rate  is difficult in practice, since it depends on properties of the function, such as the strong convexity constant,  which may not be known. Or, in the case of non-convex problems, non-existent.  Hence, the rate can be slower than optimal.  In practice, the learning rate is tuned to the problem at hand, or adaptive methods are used \cite{duchi2011adaptive, hinton2012rmsprop,zeiler2012adadelta, kingma2014adam,wu2018wngrad}. These methods often perform well in practice, but they lack optimal rates of convergence or, in some cases, they lack convergence guarantees \cite{reddi2018convergence}. In the deep learning setting,  SGD with a hand tuned learning rate schedule tends to outperform these methods in terms of generalization \cite{hardt2015train, wilson2017marginal}. 

However, while the rate of convergence $\bO(1/k)$ cannot be improved, the rate constant can be.  Improving the rate constant can be worthwhile if it leads to faster convergence in practise, as we demonstrate below.   In this work, we study a learning rate formulation  based on a generalization of Polyak's learning rate \cite[Chapter 5.3.]{polyak1987introduction} to the stochastic setting. Polyak's learning rate (see also \cite[Page 204]{beck2017first}) is commonly used for the subgradient method.  It is defined as
\begin{equation}\label{PolyakSub}
x_{k+1} = x_ k -h(x_k)\partial f(x_k), 
\qquad
h(x_k) =   \frac{f(x_k)-f^*}{||\partial f(x_k)||^2}.
\end{equation} 
In this case, the learning rate depends on an estimate of the value of $f^* = \min_{x} f(x)$.   In some applications, the $f^*$ value is known (or zero), and the method can be applied without estimation.  A principled method for estimation, which is provably convergent with the same rate, is provided in~\cite[Chapter 4.2]{boyd2013subgradient}.  The estimation involves an auxiliary sequence $\gamma_k$, leading to the learning rate
\begin{equation}
h^{est}( x_ k) = \frac{f(x_k)-f_k^{best}+\gamma_k}{||\partial f(x_k)||^2},
\end{equation} 
where $f_k^{best} = \min_{j\leq k} f(x_j)$. Convergence follows provided  
\[
\gamma_k>0, \qquad 
\gamma_k\rightarrow 0, \qquad 
\sum_{k=1}^{\infty} \gamma_k = \infty.
\]
In particular, $f_k^{best} \rightarrow f^*$ as $k\rightarrow\infty$  in this case. 
Additional methods for estimating $f^*$ can be found in \cite{shor2012minimization,boyd2013subgradient} and \cite{kim1990variable}.  In high-dimensional machine learning problems,  $f^*$ can be considered a hyper-parameter that requires tuning - not difficult in general. For example, one can observe the minimum value achieved,  $f_{min}$ on a previously trained model and retrain using Polyak's learning rate, assuming $f^* = 0.9*f_{min}$. 

The above method obtains a better constant for the rate provided the initial point is not too far from the true solution, see \eqref{eq: comp} in section \ref{sec: comparison}.  We demonstrate that the method is effective in practice, using small scale and deep learning examples in section \ref{sec : dp}. While we do not prove any results in the non-convex case, the Polyak schedule for SGD can still be applied.  

Polyak's SGD main advantage is requiring estimation of only the parameter $f^*$ which is simpler to approximate than the strong convexity constant necessary for the $\bO(1/k)$ learning rate.  Note that because the variance of the stochastic gradient is estimated as the algorithm runs, the only hyper-parameter is $f^*$.  For scheduled SGD, in a non-convex problem, it is not clear what the strong convexity constant should be (even if it somehow exists locally). However, for Polyak, in a non-convex setting the minimum of the function and the second moment of the gradients remain clear and well defined. In fact, a recent paper achieved good empirical results but without any associated theory \cite{rolinek2018l4}. 

Finally, note that throughout this paper all the rates derived are non-asymptotic, of the form
\begin{equation}
\label{eq: n-asymp}
\E{ ||x_k -x^*||}\leq \frac{C}{k+\gamma},
\end{equation}
for $C, \gamma>0$. Hence, at each iteration, the derived bound holds for $k>0$ and not only when $k\to\infty$.  In order to obtain a non-asymptotic rate for optimally scheduled SGD it is necessary to know the distance between the initial point and the true point, as this value is used in the learning rate schedule. If this value is unknown, we risk selecting an $h_k$ too large or too small, which, in turn, will prevents us from achieving the rate. On the other hand, for Polyak the non-asymptotic rate is achieved regardless as $f^*$ and $x^*$ are intrinsically connected. In \cite{bottou2016optimization}, for example, the rate is asymptotic, of the form 
\begin{equation}
\E{f(x_k)-f^* } \leq \frac{\max\{C, (\gamma + 1)(f(x_0)-f^*)\}}{k+\gamma}
 \end{equation}
 for $C,\gamma>0$.  Here, the rate constant is $C$ and does not depend on the starting point, $x_0$. However, the rate constant $C$ is not necessarily achieved in this formulation, as opposed to the non-asymptotic bound, \eqref{eq: n-asymp}. In section \ref{sec: comparison}, Figure \ref{fig: rate}  displays how our bounds compare to some test runs.

\subsubsection*{Contents}

We start this paper by introducing some mathematical background. Then, in Section~\ref{sec: adg}, which is provided for background and is not needed for the sequel, we prove a convergent rate for gradient descent using Polyak's learning rate. In Section~\ref{sec: asgd}, we recall SGD with optimally scheduled learning rate and its properties. Then we establish a convergence rate of $\bO(1/k)$, although with different constants, for optimally scheduled SGD and SGD with Polyak's learning rate, and compare.   Finally, in Section \ref{sec: numerics}, we present numerical results. We show both results for generated mini-batch noise and for an image recognition problem in deep learning.

\subsection{Notation and Convex Function Inequalities}
In this section, we recall some definitions and establish notation.
Write,
\[
f^* = \min_x f(x), \qquad
x^* \in \argmin_x f(x),
\]
when such quantities are defined. Write $g(x) = \nabla f(x)$ and $g_k = \nabla f(x_k)$.  Write as well $q(x) = \frac 1 2 |x-x^*|^2$ and  $q_k = \frac 1 2 |x_k-x^*|^2$.

The following definitions can be found in \cite{polyak1987introduction} and  \cite[Chapter 5]{beck2017first}. 

\begin{definition}
	The function $f:\R^d \to \R$ is $\mu$-strongly convex if 
	\begin{equation}\label{mu_convexori}\tag{$\mu$-convex}	
	f(x) - (f(y) + \nabla f(y) \cdot(x-y))  \geq \frac{\mu}{2}|x-y|^2,  \qquad \mbox{$x,y\in\R^d$}.
	\end{equation}
	The function  $f:\R^d \to \R$ is $L$-smooth if 
	\begin{equation}\label{L_smooth}\tag{L-smooth}	
   f(x) - (f(y) - \nabla f(y)\cdot (x-y)) \geq \frac{1}{2L} |\nabla f(x) - \nabla f(y)|^2, \qquad \mbox{ $x,y\in\R^d$}.
	\end{equation}
\end{definition}
We will use the following  inequalities in the sequel:
	\begin{align}
\label{thm217}
&f(x) - f^* \geq  \mu q(x),  & \mbox{$x\in\R^d$}\\
\label{mu_convex}
&f(x^*) - (f(x_k) + g_k\cdot (x^* - x_k))  \geq   \frac{\mu}{2} q_k, &\mbox{$x_k\in\R^d$}\\
\label{grad_to_f}
&2L (f(x) - f^*) \geq |\nabla f(x)|^2, &\qquad \mbox{ $x\in\R^d$}\\
\label{q_Lip}
&	q(x) - q(y) = (x-y)\cdot(y-x^*) + \frac 1 2 |x-y|^2, & x,y \in \R^d.
	\end{align}
 \begin{proof}
The first three inequalities follow from~\eqref{mu_convexori}, with choice $y=x^*$ and with choices $x = x^*$ and $y = x_k$, and from \eqref{L_smooth}, with choice $y=x^*$, respectively. 
To establish \eqref{q_Lip} directly, use the algebraic identity $a^2 - b^2 = (a+b)(a-b)$ applied to $q(x)-q(y)$ to obtain
		\begin{align}
			q(x) - q(y) &= \frac 1 2 (x+ y - 2x^*)\cdot(x-y)\\
			&= (x-y)\cdot \left (\frac {x+ y} 2  - x^*\right )\\
			& =(x-y)\cdot (y-x^*) + \frac 1 2 |x-y|^2. \qedhere
		\end{align}
	\end{proof}

\section{Gradient Descent with Polyak's learning rate}
\label{sec: adg}
In this section, we prove a convergence rate for Polyak's learning rate in the full gradient case.  Polyak's learning rate \cite{polyak1987introduction} is most often used in the subgradient case, as discussed above.   This result is  not needed in the sequel, but it is included for context, and because the proof is a simplified version of the proof in the stochastic case. 

Define the Polyak $x$-dependent learning rate by 
\begin{equation}
	\label{ALR}\tag{PLR}
	 h(x) = 2\frac{f(x) - f^*}{\|\nabla f(x)\|^2}.
\end{equation}
 The adaptive gradient descent sequence is given by
\begin{align}\label{AGD}\tag{AGD}
	x_{k+1}  &= x_k - h_k  \nabla f(x_k),
\end{align}
where $0 < h_k  = h(x_k)$. 
Next, we prove convergence of the iterates $x_k$ to $x^*$ in terms of the difference squared,  $q(x) = \frac 1 2 |x-x^*|^2$.
\begin{lemma}
	\label{lem: AGD}
Suppose that $f(x)$ is $\mu$-strongly convex and $L$-smooth. Let $x_k,\ h_k$ be the sequence given by \eqref{AGD} and \eqref{ALR}.  Then,
\begin{equation}\label{time_batch}
\frac 1 L \leq h_k \leq \frac 1 \mu
\end{equation}
and 
\begin{equation}
	q_k \leq (1-\mu/L)^k q_0.
\end{equation}
\end{lemma}

\begin{proof}[Proof of Lemma~\ref{lem: AGD}]
The inequality \eqref{time_batch} follows from \eqref{grad_to_f} and \eqref{thm217}.
Using the  identity  \eqref{q_Lip} with $x_{k+1}$ and $x_k$ we obtain,
\begin{align}
		q_{k+1} - q_k &= (x_{k+1}-x_k)(x_k-x^*) + \frac 1 2 (x_{k+1}-x_k)^2 && \text{by \eqref{q_Lip}}\\
		&= -h_k g_k(x_k-x^*) + \frac 1 2 h_k^2 g_k^2 && \text{by \eqref{AGD}}\\
		&\leq -h_k \mu q_k -h_k(f(x_k)-f^*) + \frac 1 2 h_k^2 g_k^2 && \text{by \eqref{mu_convex}}\\
		& \leq -h_k \mu q_k && \text{by \eqref{ALR}}\\
		& \leq - \mu q_k /L && \text{by \eqref{time_batch}} 
		\qedhere
\end{align}
\end{proof}

One outcome of the Polyak's learning rate \eqref{ALR} is that the usual restriction $h_k \leq 1/L$ on the learning rate is relaxed to $h_k \leq 1/\mu$.  For example, on a quadratic $f(x) = (\mu x_1^2 + L x_2^2)/2$, whenever $x_2 = 0$, we have $h(x) = 1/\mu$.  More generally, Figure \ref{fig: timestep} illustrates $h(x)$. Clearly from \eqref{ALR}, smaller gradients allow for larger learning rates and vice-versa.
\begin{figure}
	\label{fig: timestep}
	\centerline{
		\includegraphics[width=9cm]{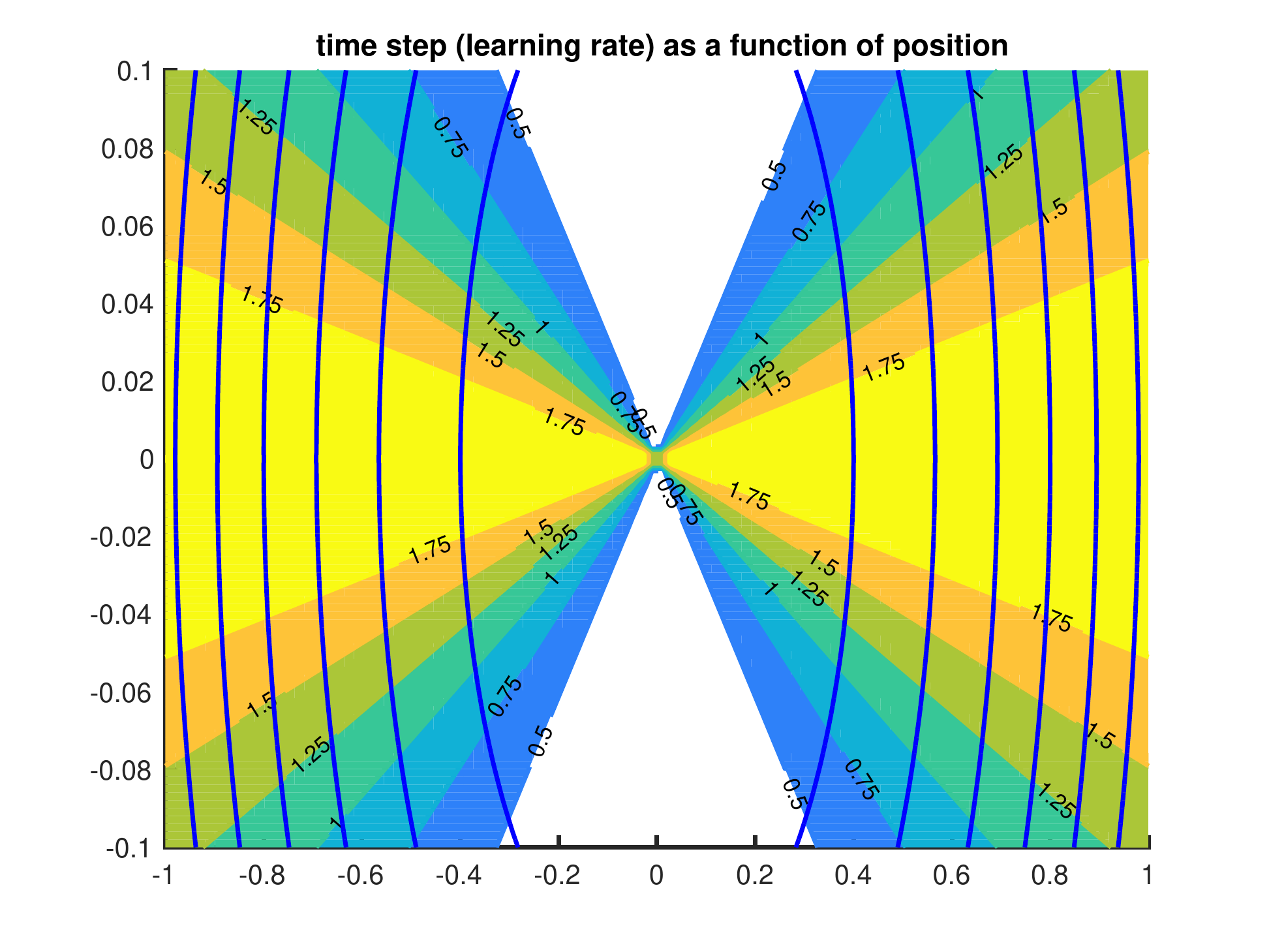}
	}	
	\caption{Level sets of $f$ superimposed on level sets of $h(x) = 2 (f(x)-f^*)/\nabla f(x)^2$. }
\end{figure}

\section{Stochastic Gradient Descent with Polyak's Learning Rate}
\label{sec: asgd}
Now we consider stochastic gradient descent. For the purpose of our analysis, we use an abstract representation, where we write $\nabla f(x)$ for the full gradient and are given an approximation $\nabla_{mb} f(x_k)$ such that
\[
\nabla_{mb} f(x ) = \nabla f(x)+ e,
\]
with $e$ a random error term.

The notation $\nabla_{mb}$ is meant to suggest mini-batch SGD, the important special case where the loss is of the form, $f(x) = \sum f_i(x)$, and a mini-batch approximation is give by $f_{mb}(x) = \sum_{i \in mb} f_i(x)$.   However, in our analysis we consider a general loss function and abstract the mini-batch error into the error term. For clarity, we assume the error is random  with zero mean and finite variance,
\begin{equation}
\label{e_mean_zero}
	\E{e} = 0,\qquad \E{e^2} = \sigma^2.
\end{equation}
More general assumptions on the error are discussed in \cite{bottou2016optimization}.  
If variance reduction \cite{johnson2013accelerating} \cite{zhang2013linear} is incorporated,  it is reflected in the $\sigma^2$ term in~\eqref{e_mean_zero}.

\subsection{Scheduled and  Polyak  SGD}
Recall the SGD step is given by
\begin{align}\label{ASGD}\tag{SGD}
	x_{k+1} 
	= x_k - h_k \nabla_{mb} f(x_k),
\end{align}
where $h_k > 0 $ is the learning rate. 

In order to achieve the optimal convergence rate for scheduled SGD, the schedule makes use of $\mu$ and $q_0 = \|x_0 - x^*\|^2/2$.  The optimal scheduled convergence rate is achieved using the following schedule \begin{equation}
	\tag{SLR}\label{LR}
	h_k = \frac{1}{\mu(k+q_0^{-1} \alpha_S^{-1}) },  
\end{equation}
where $\alpha_S$ is defined below in~\eqref{alphaS}. 
The proof (in a more general setting) can be found in \cite{bottou2016optimization}; we provide a shorter proof below.

For Polyak SGD we define the stochastic Polyak learning rate as a function of $x_k$, $\sigma^2$, and $f^*$. 
\begin{align}
	\label{ASLR}\tag{SPLR}
	 h_k =  h(x_k) =  2\frac{f(x_k) - f^*}{ \E{ \|\nabla_{mb} f(x_k)\|^2}}, 
\end{align}
\begin{remark}
From \eqref{ASLR} we can conclude that increasing the mini-batch size will also increase the learning rate, since increased mini-batch size decreases the variance of the $\nabla_{mb}f(x)$. 
\end{remark}

  A similar learning rate schedule, replacing the expectation in \eqref{ASLR} with 
  $\|\nabla_{mb} f(x_k)\|^2$  was implemented effectively in \cite{rolinek2018l4}, however no convergence proof is presented. 

\begin{remark}
To achieve the optimal constant in the convergence rate for scheduled SGD using \eqref{LR} requires knowing $q_0$ and $\mu$.  The stochastic Polyak learning rate~\eqref{ASLR} requires knowing or estimating $f^*$, as well as the variance of the stochastic gradient, which can be easily approximated.  
To reduce the computational cost, we can evaluate \eqref{ASLR} every fixed number of  number of iterations (or epochs).
\end{remark}

Convergence rates for  SGD using both the optimally scheduled and Polyak learning rates are proven in Theorem~\ref{sgd_thm} and~\ref{asgd_thm} below,  respectively.

\subsection{Basic inequality}
We begin by proving an inequality which we be used for the proof of the convergence rate for both methods. 
\begin{lemma}
	\label{lem: est1}
	Suppose $f$ is $\mu$-strongly convex. Assume that \eqref{e_mean_zero} holds. 
	Let $x_{k+1}$ be   given by \eqref{ASGD} for any $h_k >0$. 	 
	Then, 
	\begin{align}
		\E{q_{k+1}\mid x_k} \leq (1-\mu h_k)q_k -h_k(f(x_k)-f^*) +\frac 1 2h_k^2( g_k^2 + \E{e^2} ).
	\end{align}
\end{lemma}

\begin{proof}

	We start by estimating the difference between consecutive iterates,
	\begin{align}
		q_{k+1} - q_k
		&= (x_{k+1}-x_k)(x_k-x^*) + \frac 1 2 (x_{k+1}-x_k)^2  &&\text{by \eqref{q_Lip}}\\
		&= -h_k (g_k+e)(x_k-x^*) + \frac 1 2 h_k^2 (g_k+e)^2 &&\text{by \eqref{ASGD}}\\
		&\leq -h_k \mu q(x_k) -h_k(f(x_k)-f^*) -h_k e(x_k-x^*) + \frac 1 2 h_k^2 (g_k+e)^2 &&\text{by \eqref{mu_convex}.} 
	\end{align}
	Then, by taking the expectations of the last inequality and conditioning on $x_k$ we obtain,
	\begin{align}
		\E{ q_{k+1}\mid x_k} & \leq  (1- \mu h_k) q_k -h_k(f(x_k)-f^*) 
		-h_k \E{e(x_k-x^*)} \\ & \hspace{150pt}
		+ \frac 1 2 h_k^2 \E{(g_k+e)^2}.
	\end{align}
	Note that, by assumption  \eqref{e_mean_zero},  it follows that
	\[
	\E{e(x_k-x^*)} = 0.
	\]
	Moreover, we may expand the last term as 
	\[
	\E{ (g_k + e)^2 } = \E{g_k^2} + 2 g_k \E{e} + \E{e^2} = g_k^2 + \E{e^2}.
	\]
	Thus,
	\begin{align}
		\E{ q_{k+1}\mid x_k}  & \leq  (1- \mu h_k) q_k -h_k(f(x_k)-f^*) +\frac 1 2 h_k^2 ( g_k^2 + \E{e^2} ),
	\end{align}
	which establishes the inequality.
\end{proof}

\subsection{Scheduled SGD Rate}
Now, we prove the rate for SGD with schedule \eqref{LR}. We use this proof structure to easily  draw a parallel between the different learning rate choices, \eqref{LR} and \eqref{ASLR}.
\begin{theorem}\label{sgd_thm}
	Suppose $f$ is $\mu$-strongly convex and $L$-smooth.  Assume \eqref{e_mean_zero}.
	Let $x_k,\ h_k$ be the sequence given by~\eqref{ASGD} with the optimal learning rate schedule~\eqref{LR}. Set $M= \max_k \|g_k\|$.
	Then,
	\begin{equation}
	\E{q_k \mid x_{k-1}} \leq \frac{1}{\alpha_S k +q_0^{-1}}, \mbox{ for all $k\geq 0$,}
	\end{equation}
	where
	\begin{equation}\label{alphaS}
	\alpha_S = \frac{2 \mu^2}{\sigma^2+M^2}.
	\end{equation}
\end{theorem}

\begin{proof}

	We prove the rate by induction. Clearly, it holds for $k=0$. For the induction step, assume it the rate holds and define $\hat{k} = k+q_0^{-1}\alpha_S^{-1}$. Then, from Lemma \ref{lem: est1}, 
	\begin{align}
		\E{q_{k+1} } 
		& \leq(1-2 \mu h_k) q_k + \frac 1 2 h_k^2 (g_k^2 + \E{e^2}) &\mbox{  by \eqref{thm217}}
		\\
		& \leq \left(1-\frac 2 {\hat{k}} \right) q_k + \frac{1}{ 2\mu^2 \hat{k}^2}(M^2 +\sigma^2) & \mbox{by \eqref{LR} and assumption}\\
		& \leq \left(1-\frac 2 {\hat{k}} \right) \frac{1}{\alpha_{S} \hat{k}}+ \frac{1}{ \alpha_{S} \hat{k}^2} & \mbox{by the induction hypothesis} \\
		&   = \left(\frac{\hat{k}-1} {\hat{k}^2} \right) \frac{1}{\alpha_{S}} \leq \frac{1}{\alpha_S} \frac{1}{\hat{k}+1} = \frac{1}{\alpha_S (k+1) + q_0^{-1}},
	\end{align}
	which proves the rate.
	\end{proof}

\subsection{Polyak SGD Rate} 

Now we prove the rate for Polyak SGD, \eqref{ASGD} using \eqref{ASLR}.
\begin{theorem}\label{asgd_thm}
Suppose $f$ is $\mu$-strongly convex and $L$-smooth.  Assume the mean and variance of $e$ are given by ~\eqref{e_mean_zero}.  Let $x_k$, $h_k$ be the sequence given by~\eqref{ASGD} and \eqref{ASLR}.  
		Then,
		\begin{align}
			\E{q_k \mid x_{k-1}} \leq \frac{1}{\alpha_P k + q_0^{-1}}, \qquad \text{ for all  $k \geq 0$ ,}
		\end{align}
		where 
		\[
		\alpha_P=\frac{2 \mu^2}{\sigma^2+ 2\mu^2(L-\mu) q_0}.
		\]		
\end{theorem}

The proof of the theorem requires two auxiliary Lemmas. Lemma \ref{lem: est2} provides an inequality for $q_{k+1}$ in terms on the previous iterate, $q_k$.  Lemma \ref{lem: transf} provides a further bound. We combine both in an induction proof at the end of the section.

\begin{lemma}
	\label{lem: est2}
	Suppose $f$ is $\mu$-strongly convex and $L$-smooth.  
	Let $x_k,\ h_k$ be the sequence given by \eqref{ASGD} and \eqref{ASLR}.  
	Assume that \eqref{e_mean_zero} holds. 
	Then, 
	\begin{align}\label{main_est}
		\E{q_{k+1}\mid x_k} \leq (1-\mu h_k)q_k \leq \frac{ \beta +r q_k }{ \beta+q_k }q_k ,
	\end{align}
	where  $r = 1 - \frac \mu L$ and $\beta = \frac{\sigma^2}{2\mu L}$.
\end{lemma}

\begin{proof}
Apply Lemma \ref{lem: est1} with $h_k$ given by \eqref{ASLR} to obtain	\begin{equation}
	\E{ q_{k+1}\mid x_k}  \leq (1-\mu h_k)q_k.
	\end{equation}
We have the following estimate on the learning rate,
\begin{align}
	h_k &= 2 \frac{f_k - f^{\star}}{|g_k|^2 + \sigma^2}  & \mbox{by \eqref{ASLR}} \\
	& \geq 2 \frac{f_k - f^{\star}}{2 L (f_k - f^*) + \sigma^2} & \mbox{by  \eqref{grad_to_f}} \\
	& = \frac{2}{2 L + \frac{\sigma^2}{f_k -f^*}} & \mbox{ divide by $f_k - f^*$} \\
	& \geq \frac{2}{2 L + \frac{\sigma^2}{\mu q_k}} & \mbox{ by \eqref{thm217}}.\\
\end{align}
This establishes the second inequality in \eqref{main_est} as follows,
\begin{align}
	(1 - \mu h_k) q_k &\leq  \left( 1 - \frac{2\mu}{2 L + \frac{\sigma^2}{\mu q_k}}  \right) q_k \leq  \frac{2L q_k +\frac {\sigma^2} \mu - 2\mu q_k}{2 L + \frac{\sigma^2}{\mu q_k}} & \\
	& =  \frac{(1 - \frac \mu L) q_k +\frac{ \sigma^2}{2 \mu L}}{q_k+ \frac{\sigma^2}{2\mu L}}q_k &\\
	&  = \frac{r q_k + \beta}{q_k + \beta}q_k, 
\end{align}
where $r = 1 - \frac \mu L$ and $\beta = \frac{\sigma^2}{2\mu L}$.
\end{proof}

Next, we establish the following technical, but elementary Lemma, which is used to prove the rate. 
\begin{lemma}
	\label{lem: transf}
Define the transformation $T$,
\[
 T(x) =  \frac{\beta + rx}{\beta + x}x,
 \] 
for $\beta>0$ and for $r \in [0,1)$.  
Given $b > 0$ define
\[
c =\frac{1-r}{\beta+r/b}.
\]

Then, for  all $k\geq 0$, 
\[
 T\left(\frac{1}{c k + b}\right) \leq \frac{1}{c (k+1) +b}.
 \]

\end{lemma}
\begin{proof}
Consider $x = \frac{1}{y}.$
Then, define $S$ as the multiplicative inverse of $T$: 
\begin{align}
	\label{eq: S}
	T(x) = T\left(\frac 1 y\right) = \frac 1 y \frac{\beta y + r}{\beta y +1} := \frac{1}{S(y)}.
\end{align}

	Next, we consider the sequence $a_k := c k + b$. We start by expanding the difference:
	\begin{align}
		S(a_k) - a_{k+1} &= a_k\frac{\beta a_k + 1}{\beta a_k + r} - (a_k + c) \\
		&= \frac{a_k(1 - r - c\beta) - c r}{\beta a_k + r} \\
		& =  \frac{k c(1 - r - c\beta) + (b(1-r-c\beta)- c r)}{\beta ck + (\beta b+ r)} .
	\end{align}
	
	This difference is positive if both the denominator and numerator are positive for $k\geq 0$. The denominator's positivity follows from $b,c,r,\beta \geq 0$. The numerator is positive for all $k\geq 0$ provided
	\begin{align}
		1 - r - c\beta \geq 0,
	\end{align}
	since the choice of $c$ guarantees $b(1-r-c\beta)- c r =0.$
	In fact, positivity of this expression also holds for our choice of $c$. 
	
	Hence, $S(a_k) - a_{k+1 } \geq 0$. Then, using \eqref{eq: S},
	\[
T\left(\frac 1 {a_k}\right)  =  \frac{1}{S(a_k)}\leq \frac 1 {a_{k+1}} = \frac 1 {c(k+1) +b},
	\]
and	the lemma follows.
\end{proof}

\begin{proof}[Proof of Theorem \ref{asgd_thm}]
We proceed by induction. 

The base case corresponds to $q_0 \leq  \frac{1}{\alpha_P k + q_0^{-1}}$ for $k=0$, which clearly holds.  

For the induction step, assume that $q_k \leq \frac{1}{\alpha_P k + q_0^{-1}}$.
We will apply  Lemma \ref{lem: transf} with  $b=q_0^{-1}$ and the corresponding $c$. 

\begin{align}
	\E{q_{k+1}|x_k} & \leq T(q_k) & \mbox{ by Lemma \ref{lem: est2} }\\
	& \leq T\left(\frac 1 {\alpha_P k + q_0^{-1}}\right) & \mbox{by monotonicity of $T$ and  hypothesis} \\
	& \leq \frac 1 {\alpha_P(k+1) +q_0^{-1}} & \mbox{ by Lemma \ref{lem: transf},  } 
\end{align}
which concludes the induction.
\end{proof}


\subsection{Comparison of the Polyak and Scheduled Rate Constants} \label{sec: comparison}

The constant in the rate established in Theorem~\ref{asgd_thm} for Polyak SGD
can be smaller than the rate for scheduled SGD in Theorem~\ref{sgd_thm}.
Recall the rates are given by,

\[
\alpha_{S} = \frac{2 \mu^2}{\sigma^2+M^2},\qquad
\alpha_{P} =\frac{2 \mu^2}{\sigma^2+ 2\mu^2(L-\mu) q_0}.
\]
Now, we note that $\alpha_{P} \geq \alpha_{S}$ if

\begin{equation}
\label{eq: comp}
q_0 \leq \frac{M^2}{2\mu^2(L-\mu)}.
\end{equation}
Since we can reset the algorithm, this condition will eventually hold, using the reset value for $q_0$. In Figure \ref{fig: rate} we can see the gap between a run of SGD with both schedules and the non-asymptotic bounds derived. Note that the bound could be recomputed at each iteration, yielding tighter ones.
\begin{figure}
	\centerline{
		\includegraphics[width=10cm]{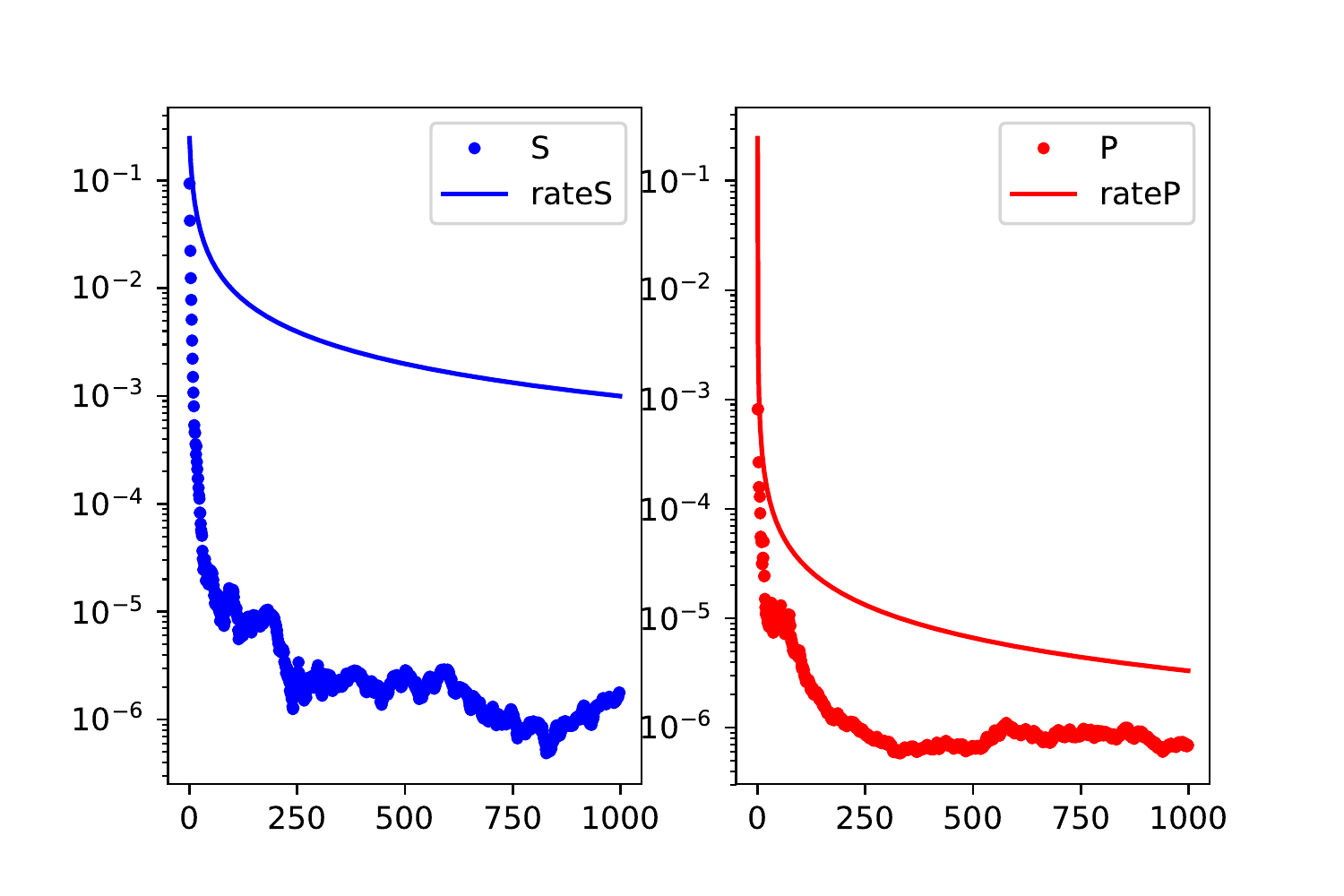}	
	}
	\caption{Average of 5 test runs of scheduled SGD with bound (blue) and Polyak SGD with bound (red) on a quadratic example. The batch size selected was 100 over 1000 sample points.}
	\label{fig: rate}
\end{figure}

\section{Numerical Results}
\label{sec: numerics}
In this section, we implement the Polyak adaptive SGD algorithm and compare it to scheduled SGD.  The Polyak SGD code is implemented in PyTorch\footnote{\hyperlink{https://github.com/marianapraz/polyakSGD}{https://github.com/marianapraz/polyakSGD}}.
For practical implementations, we capped the learning rates with a minimum and maximum value to protect against errors in the estimation of $f^*$.  

The first problem we considered was
\begin{equation}
	f(x) = \frac 1 {2N} \sum_{i=1}^N \|x - x_i\|^2,
\end{equation}
for a given data set $x_1,\dots, x_N \in \R^2$.  The stochastic gradients were obtained by taking a different random mini-batch of fixed size $M$ at each iteration.  A mini-batch of size $M$ is a random subset $I \subset [1,\dots, N]$ with $|I| = M$.  Then, 
\begin{equation}
	f_{I} = \frac 1 {2|I|} \sum_{i\in I} \|x - x_i\|^2
\end{equation}
and the stochastic gradient is the gradient of $f_{I}$.

In Figure~\ref{fig: mb1}, we illustrate the learning rate as a function of the mini-batch size and the location of the current iterate. As expected, see \eqref{ASLR}, the learning rate increases when the mini-batch noise is smaller, which happens when we have a larger mini-batch. 

\begin{figure}
	\centerline{
		\includegraphics[width=6cm]{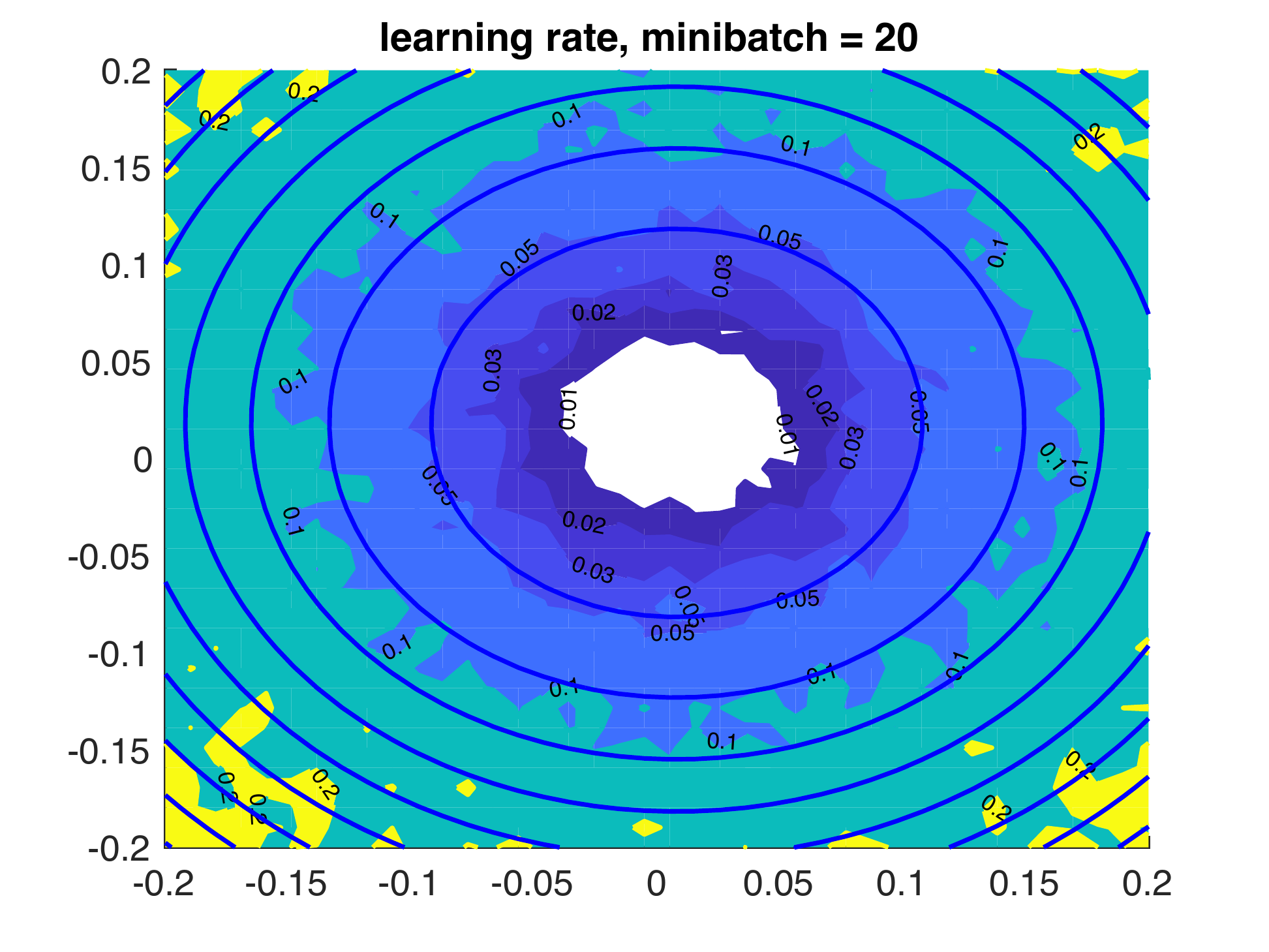}\includegraphics[width=6cm]{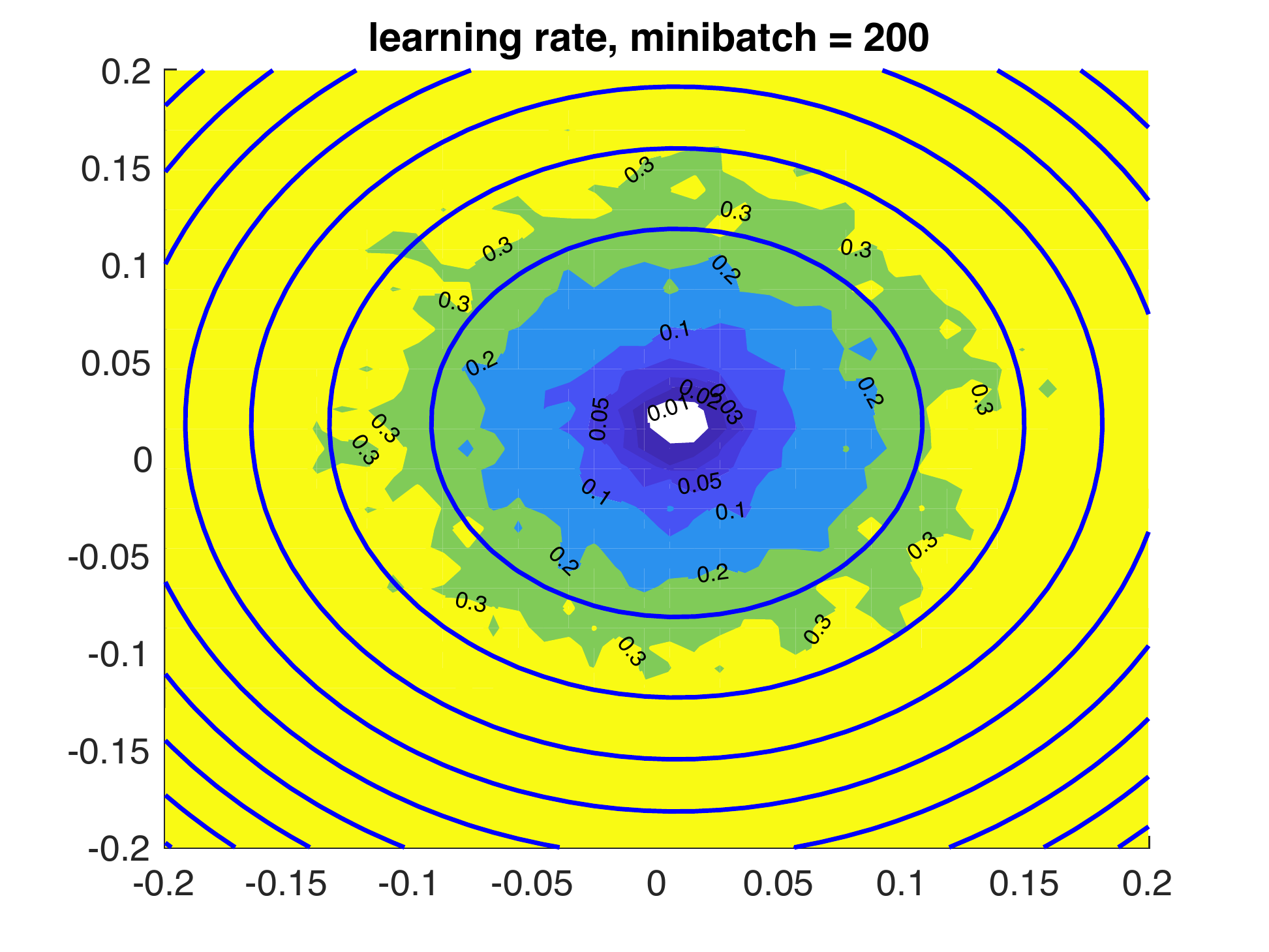}	
	}
	\caption{Level sets of $f(x,y) = x^2 +y^2 $ superimposed on level sets of formula \eqref{ASLR}. Left: mini-batch size 20. Right: mini batch size 200.}
	\label{fig: mb1}
\end{figure}

%

We see, in Figure \ref{fig: asgd1}, how Polyak SGD compares to two schedules of SGD. The first schedule is a typical deep learning schedule that reduces the learning rate every hundred steps by a fixed amount (in this case, a sixth). The second is the optimal schedule~\eqref{LR}. 

In Figure \ref{fig: asgd3}, we share a special case where we start close to an optimal value, which Polyak immediately recognizes by lowering the learning rate.

\begin{figure}
	
	\centerline{
		\includegraphics[width=6.5cm]{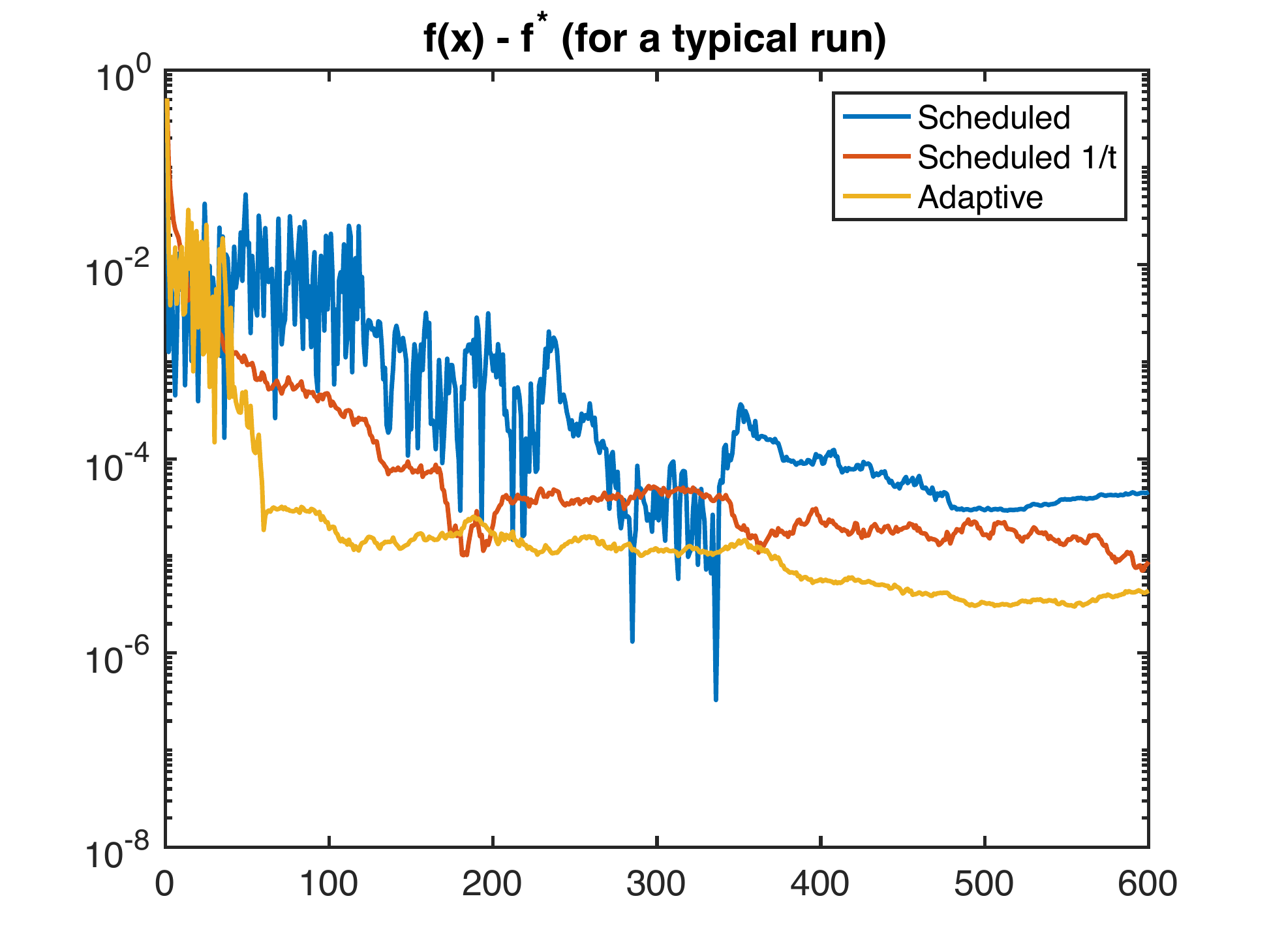}
		\includegraphics[width=6.cm]{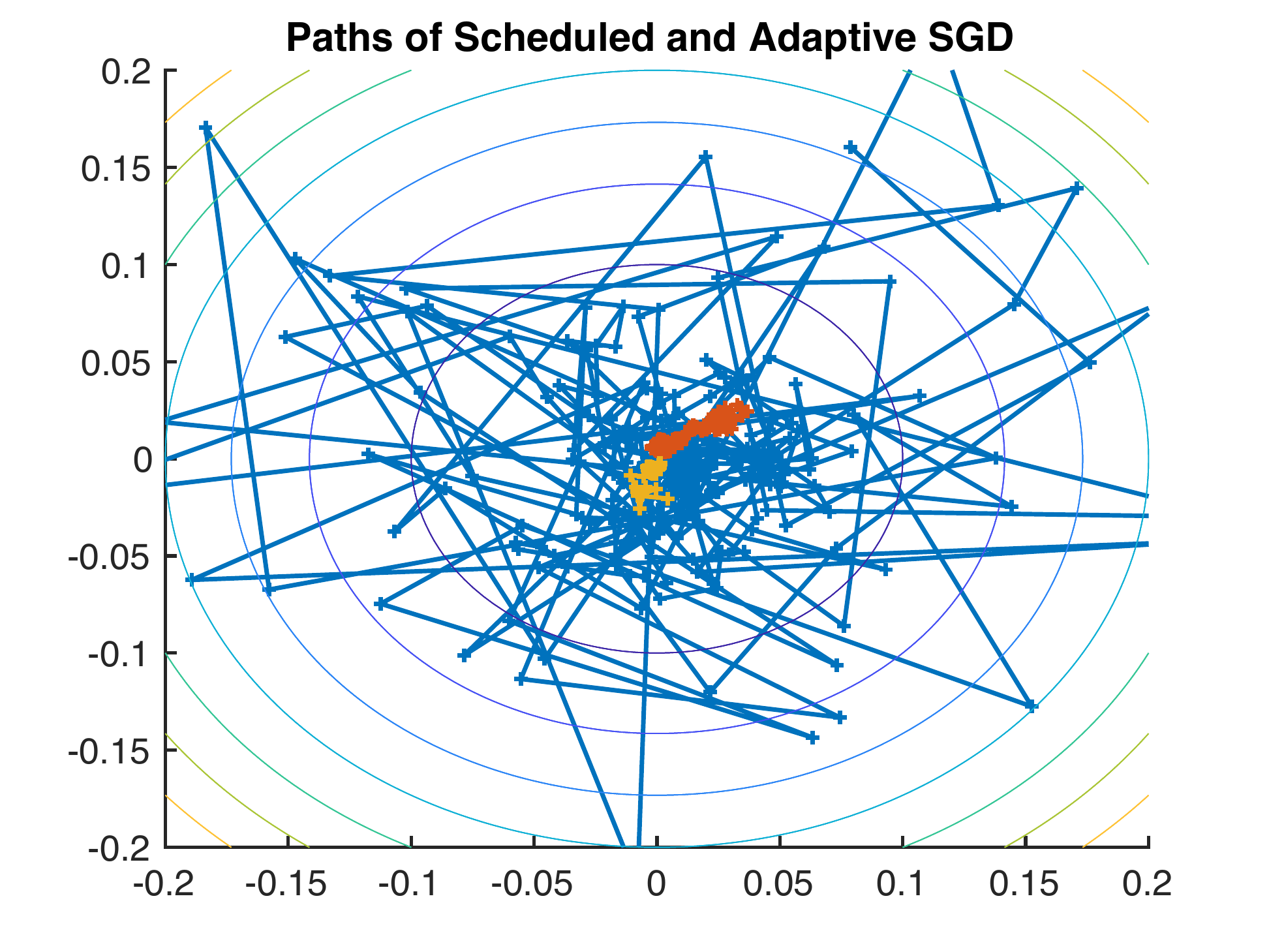}
	}
	\includegraphics[width=6.5cm]{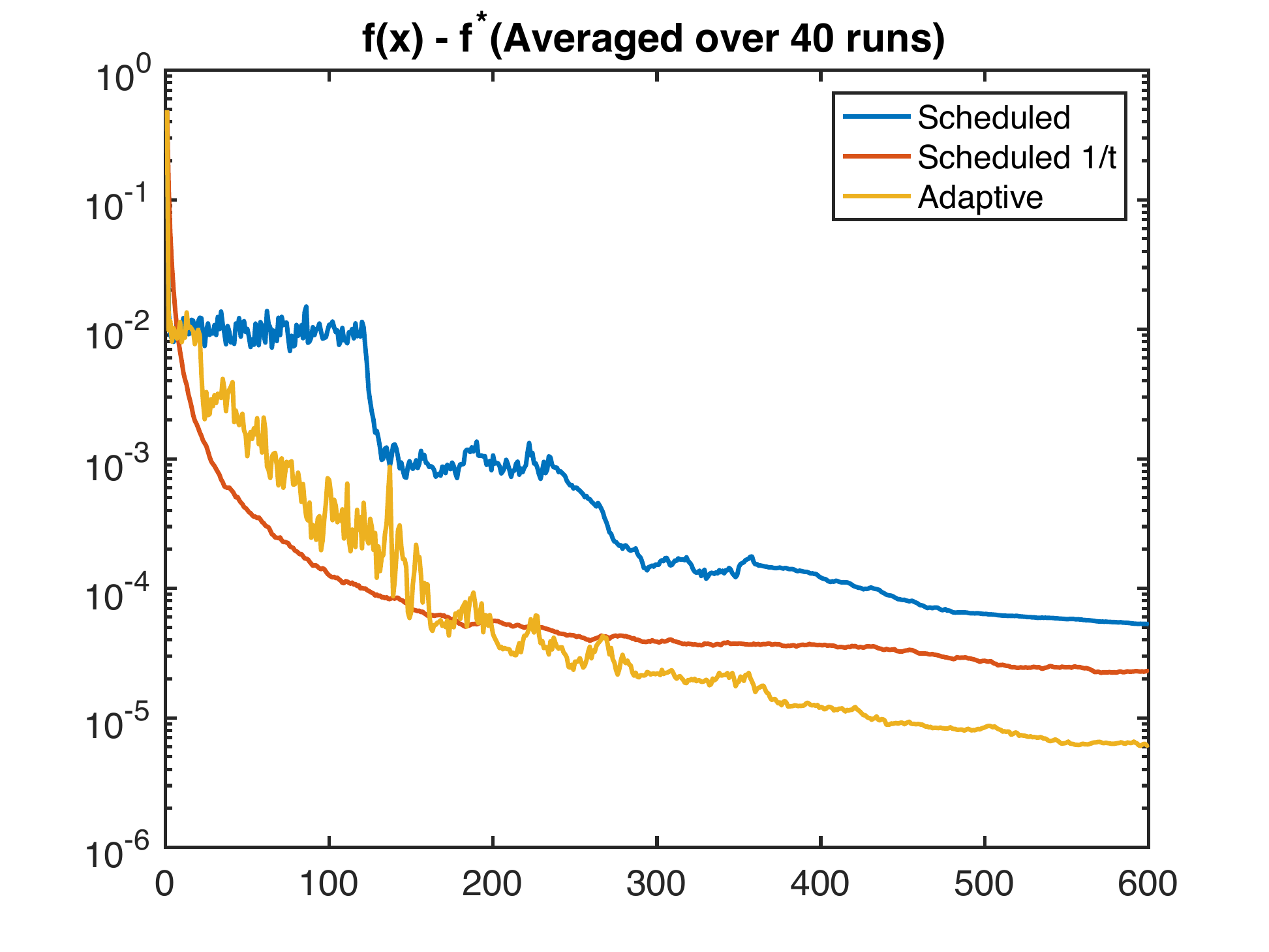}\includegraphics[width=6.cm]{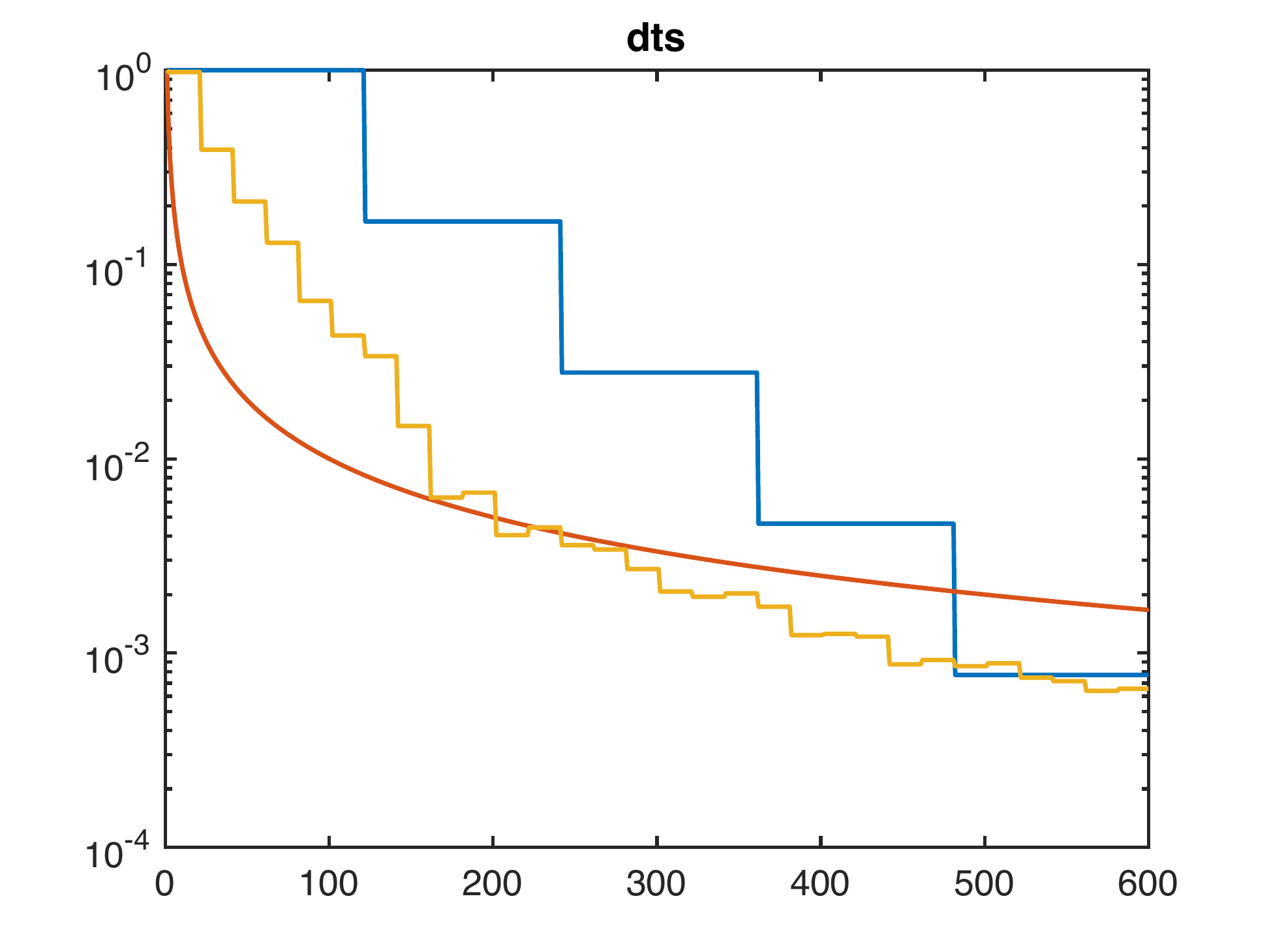}
	\caption{Comparison of SGD with Polyak (adaptive) SGD. Top left: paths of optimally scheduled SGD, epoch scheduled SGD, and Polyak (adaptive) SGD. From one path it is not clear which algorithm is faster. Bottom left: average excess loss over 40 runs of each algorithm.  Now it is clear that the Polyak algorithm is faster on average.  Top Right: plots of a path. Bottom right: illustration of the learning rates computed. 	
	}
	\label{fig: asgd1}
\end{figure}

%
%


\begin{figure}
	
	\centerline{
		\includegraphics[width=5cm]{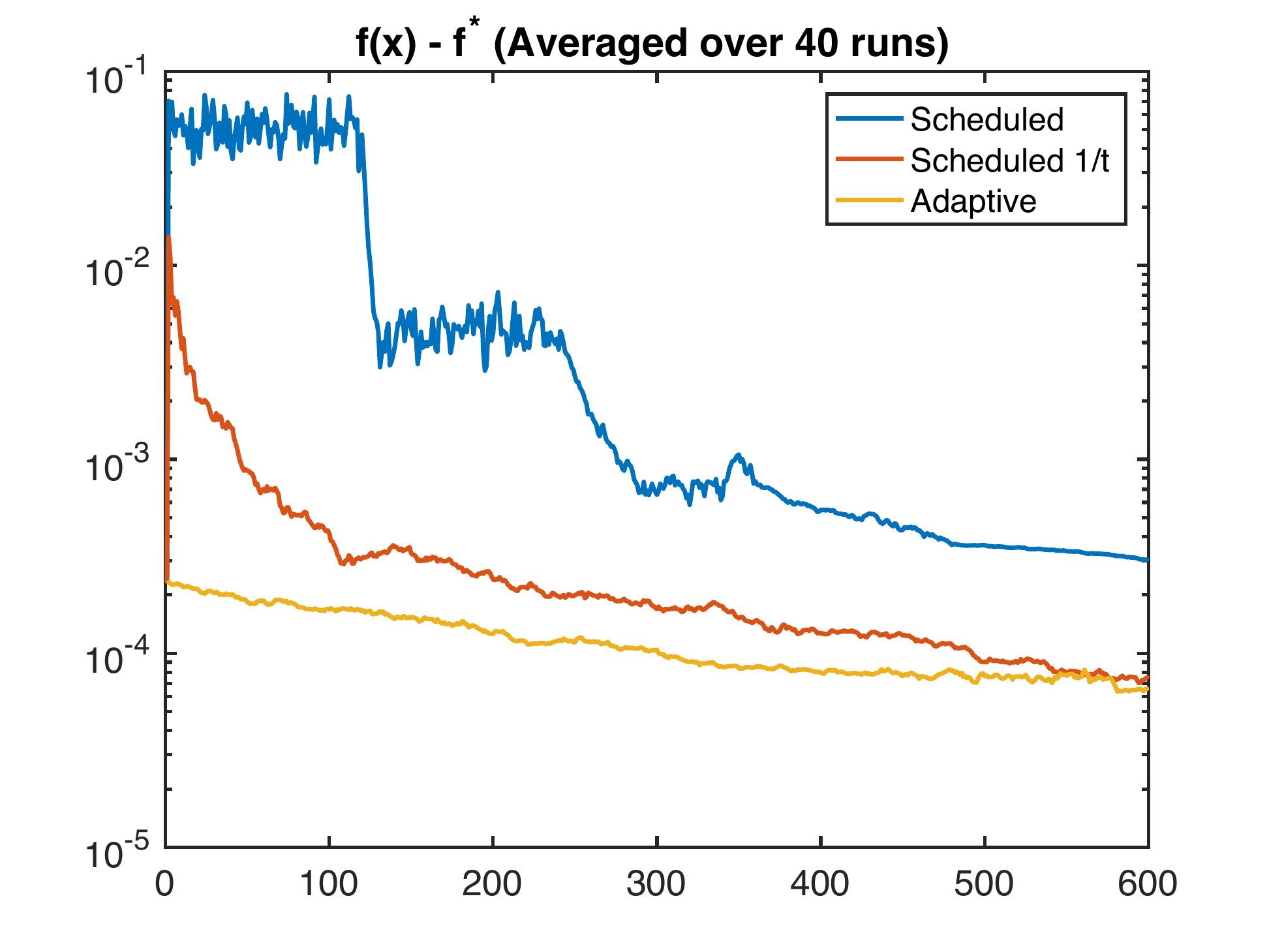}\includegraphics[width=5cm]{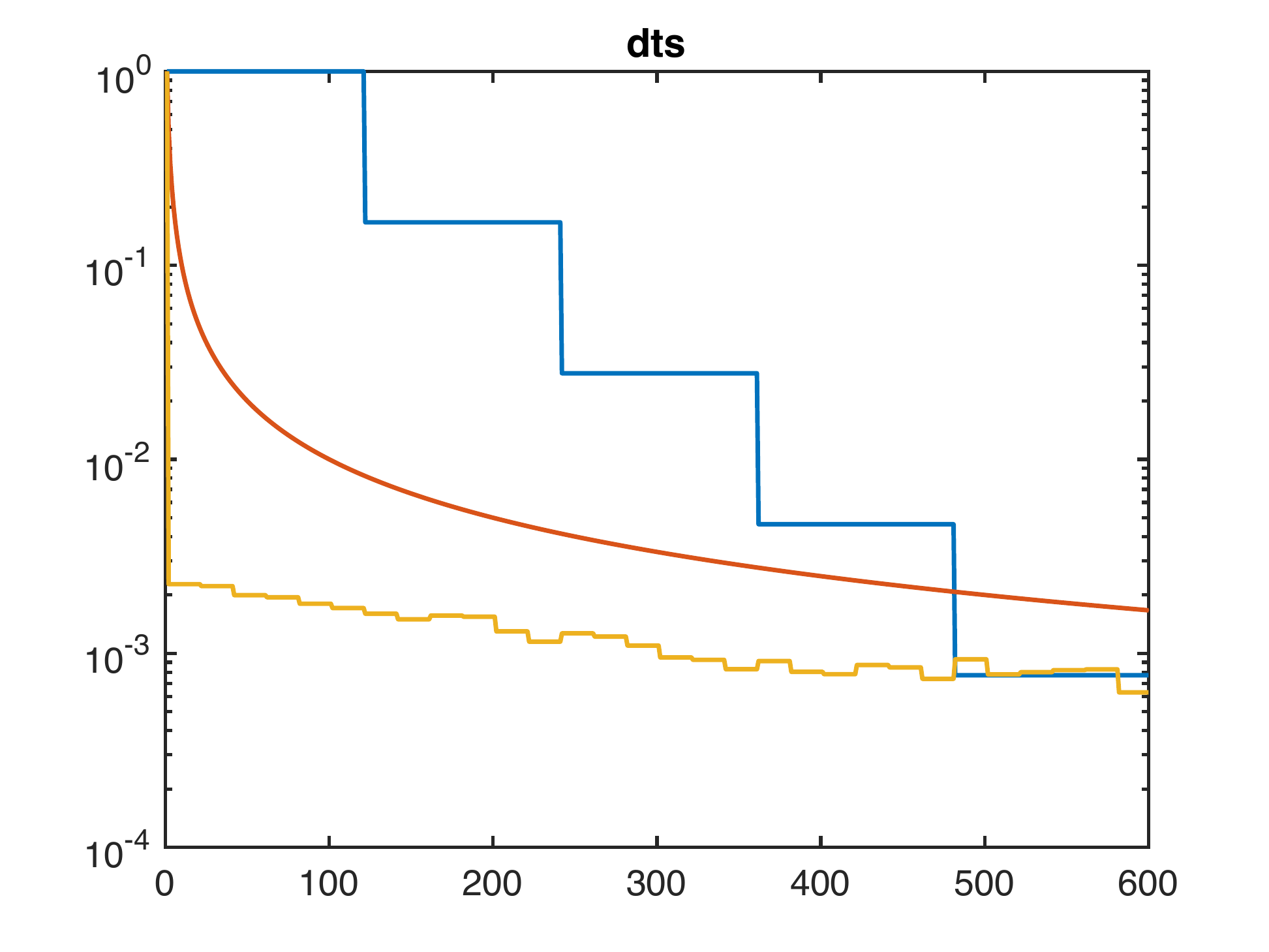}
	}
	\caption{Comparison of SGD with Polyak (adaptive) SGD. 
		In the case where the initial value is close to optimal, (non-optimally) scheduled SGD forgets the good initialization.  On the other hand, the Polyak method detects the good initialization and improves the values. }
\label{fig: asgd3}
\end{figure}


\subsection{Deep Learning Example}
\label{sec : dp}
We trained an AllCNN architecture, \cite{springenberg2014striving}, for image classification on CIFAR-10. Our baseline was trained using SGD and the schedule we used reduces the learning rate by a factor of five every 60 epochs. We used no momentum and no regularization.
In Figure \ref{fig: cifar10}, we see that we obtain similar results in minimizing the training loss without impacting the testing error (the most relevant metric for the success of a neural net training). The schedule of SGD for the CIFAR-10 dataset has already been tuned very well, so it performs similarly to Polyak SGD, which required no tuning. 
The estimate of $f^*$ was obtained with one run of scheduled SGD.  In practice, networks are trained many times using similar parameters, so Polyak SGD could provided some advantages in training time when averaged over many runs. 

\begin{figure}
	
	\centerline{
		\includegraphics[height=6cm]{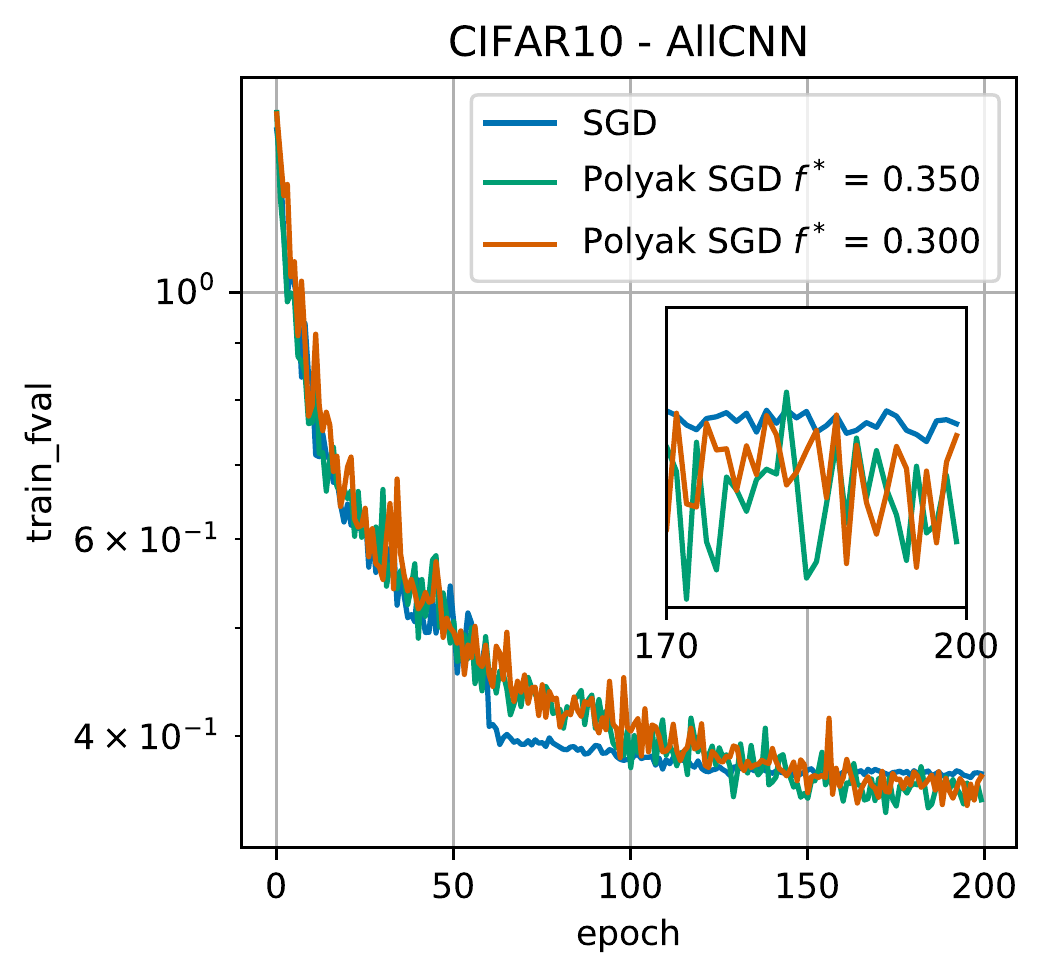}\includegraphics[height=6cm]{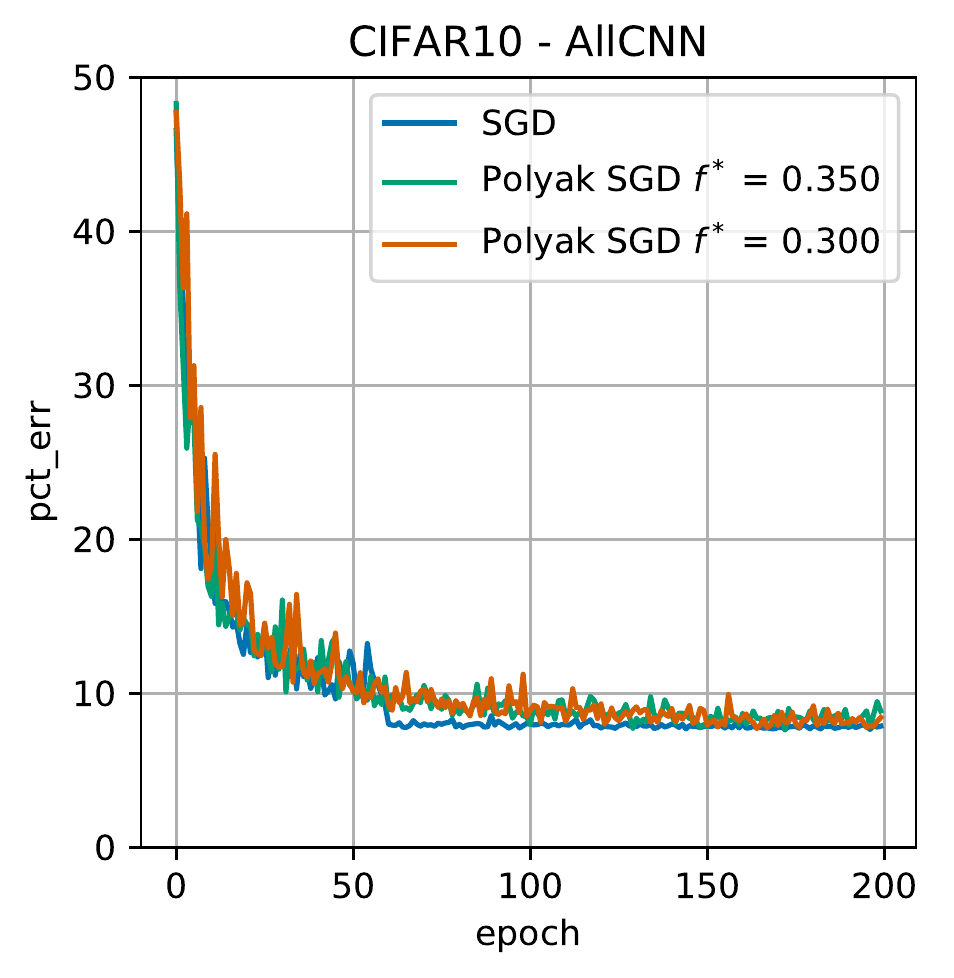}
	}
	\caption{Comparison of scheduled SGD with Polyak SGD. The first plot is the training loss and the second one is the test error.}
	\label{fig: cifar10}
\end{figure}

\bibliographystyle{alpha}
\bibliography{SGD2}

\end{document}